\documentclass[12pt,reqno]{amsart}
\usepackage{fullpage}
\usepackage{times}
\usepackage{amsmath,amssymb,amsthm,url}
\usepackage[utf8]{inputenc}
\usepackage[english]{babel}
\usepackage{bbm}
\usepackage{enumerate}
\usepackage{bm}
\usepackage{graphicx}
\usepackage{mathrsfs}
\usepackage[colorlinks=true, pdfstartview=FitH, linkcolor=blue, citecolor=blue, urlcolor=blue]{hyperref}
\usepackage{tikz-cd}
\usepackage{tabstackengine}
\stackMath
\usepackage{comment}
\usepackage[capitalise]{cleveref}

\newtheorem{thm}{Theorem}[section]

\newtheorem{prop}[thm]{Proposition}
\newtheorem{cor}[thm]{Corollary}

\newtheorem{claim}[thm]{Claim}

\theoremstyle{definition}

\newtheorem{rem}[thm]{Remark}

\newenvironment{poc}{\begin{proof}[Proof of claim]}{\end{proof}}

\newcommand{\F}{\mathbb{F}}

\newcommand{\Cay}{\operatorname{Cay}}

\title{Van Lint--MacWilliams' conjecture and maximum cliques in Cayley graphs over finite fields, II}
\author{Chi Hoi Yip}
\address{School of Mathematics\\ Georgia Institute of Technology\\ Atlanta, GA 30332\\ United States}
\email{cyip30@gatech.edu}
\begin{document}

\begin{abstract}
The well-known Van Lint--MacWilliams' conjecture states that if $q$ is an odd prime power, and $A\subseteq \F_{q^2}$ such that $0,1 \in A$, $|A|=q$, and $a-b$ is a square for each $a,b \in A$, then $A$ must be the subfield $\F_q$. This conjecture was first proved by Blokhuis and is often phrased in terms of the maximum cliques in Paley graphs of square order. Previously, Asgarli and the author extended Blokhuis' theorem to a larger family of Cayley graphs. In this paper, we give a new simple proof of Blokhuis' theorem and its extensions. More generally, we show that if $S \subseteq \F_{q^2}^*$ has small multiplicative doubling, and $A\subseteq \F_{q^2}$ with $0,1 \in A$, $|A|=q$, such that $A-A \subseteq S \cup \{0\}$, then $A=\F_q$. This new result refines and extends several previous works; moreover, our new approach avoids using heavy machinery from number theory.
\end{abstract} 

\maketitle

\section{Introduction}
Throughout the paper, $p$ denotes a prime, and $q,r$ denote positive powers of $p$. We use $\F_r$ to denote the finite field with $r$ elements. For a finite field $\F_r$, let $\F_r^*=\F_r\setminus \{0\}$ be its multiplicative group. A {\em clique} in a graph $X$ is a subset of vertices of $X$ that is pairwise adjacent, and a {\em maximum clique} is a clique with the maximum size. For a graph $X$, the {\em clique number} of $X$, denoted $\omega (X)$, is the size of a maximum clique. 

In 1978, van Lint and MacWilliams \cite{vLM78} studied the vectors of minimum weight in certain generalized quadratic residue codes. They proposed the following conjecture: if $q$ is an odd prime power, $A$ is a subset of $\F_{q^2}$ such that $0,1 \in A$, $|A|=q$, and $a-b$ is a square in $\F_{q^2}$ for each $a,b \in A$, then $A$ is necessarily the subfield $\F_q$. This conjecture was first confirmed by Blokhuis \cite{Blo84} in 1984. Blokhuis' theorem is often formulated in terms of the characterization of maximum cliques in Paley graphs of square order. Next, we introduce some standard terminologies related to Cayley graphs and generalized Paley graphs. 

Let $G$ be an abelian group and $S$ be a subset of $G\setminus \{0\}$ with $S=-S$, the \emph{Cayley graph} $X=\Cay(G;S)$ is the graph with vertex set $G$, such that two vertices are adjacent if and only if their difference is in the \emph{connection set} $S$. Note that $A\subseteq G$ is a clique in $\Cay(G;S)$ if and only if $A-A \subseteq S \cup \{0\}$. Generalized Paley graphs are Cayley graphs defined on the additive group of a finite field $\F_r$, where $S$ is a (multiplicative) subgroup of $\F_{r}^*$. More precisely, let $d>1$ be an integer and $r \equiv 1 \pmod {2d}$. The {\em $d$-Paley graph} on $\F_r$, denoted $GP(r,d)$, is the Cayley graph $\operatorname{Cay}(\F_r;(\F_r^*)^d)$, where $(\F_r^*)^d=\{x^d: x \in \F_r^*\}$ is the set of $d$-th powers in $\F_r^*$. The classical Paley graphs are $2$-Paley graphs. Note that the assumption $r \equiv 1 \pmod {2d}$ ensures that $-1\in (\F_r^*)^d$, so the graph $GP(r,d)$ is well-defined. While generalized Paley graphs are often only defined over $\F_r$ with $r$ odd, one can also extend the definition to $\F_r$ with $r$ even. When $r$ is even, any subset $S\subseteq \F_r^*$ satisfies $S=-S$, so whenever $d>1$ with $d\mid (r-1)$, we can define $GP(r,d)=\operatorname{Cay}(\F_r;(\F_r^*)^d)$. For a generalized Paley graph $GP(r,d)$, the trivial upper bound on its clique number is $\sqrt{r}$ \cite[Theorem 1.1]{Y22}. Even improving this trivial upper bound by $1$ requires nontrivial efforts. It is a major open question in arithmetic combinatorics (as well as in analytic number theory and extremal combinatorics) to estimate the clique number of generalized Paley graphs; we refer to recent progress in \cite{DSW, HP, KYY24b, Y22, Yip1, Y25+}. 

Using the above terminology, Blokhuis' theorem states that if $q$ is an odd prime power, then the trivial upper bound on the clique number of the Paley graph $GP(q^2,2)$ is tight; moreover, the only maximum clique in $GP(q^2,2)$ containing $0,1$ is $\F_q$. Blokhuis' proof is based on an ingenious polynomial method. From his result, it is easy to deduce that all maximum cliques in $GP(q^2,2)$ are of the form $a\F_q+b$, where $a$ is a nonzero square in $\F_{q^2}$ and $b\in \F_{q^2}$: to show that $a\F_q+b$ is a clique, it suffices to check that $a\F_q$ is contained in the set of squares in $\F_{q^2}$. Since maximum cliques arising from the subfield $\F_q$ are the only obvious maximum cliques in $GP(q^2,2)$, Blokhuis' theorem is also known as an analogue of the Erd\H{o}s--Ko--Rado (EKR) theorem for Paley graphs of square order; we refer to \cite[Section 5.9]{GM15} and \cite{AY22, GY24, NM, Y22, Y24} for further discussions. 

Later, in 1991, Bruen and Fisher \cite{BF91} gave a different proof of Blokhuis' theorem using the so-called ``Jamison method". In 1999, Sziklai \cite{Szi99} extended Blokhuis's theorem to generalized Paley graphs $GP(q^2,d)$ with $d \mid (q+1)$: the only maximum clique containing $0,1$ is $\F_q$. The methods used in \cite{Blo84, BF91, Szi99} are algebraic in nature and highly rely on the fact that the connection set of generalized Paley graphs $GP(q^2,d)$ with $d \mid (q+1)$ are subgroups of $\F_{q^2}^*$ and can be written as a union of $\F_q^*$-cosets in $\F_{q^2}^*$. As pointed out in \cite[Chapters 3 and 8]{NM} and \cite[Section 2]{AY22}, these proof techniques cannot be extended to general Cayley graphs. 

In the previous paper \cite{AY22} in 2022, Asgarli and the author presented a new approach and extended Blokhuis' theorem to a family of Cayley graphs $\Cay(\F_{q^2}; S)$, where $\F_q^* \subseteq S$ and $S$ is a union of at most $\frac{q+1}{2}$ many $\F_q^*$-cosets in $\F_{q^2}^*$. The proof of the main result \cite[Theorem 1.3]{AY22} relied on a recent result on character sum estimates over affine spaces due to Reis \cite{Reis}, which crucially relied on a deep result by Katz \cite{Katz} that extended the celebrated Weil bound on complete character sums over finite fields \cite{W48} to character sums over affine lines. Due to the analytic nature of the proof, this result requires the following two technical assumptions:
\begin{enumerate}
    \item there exists a real number $\epsilon>0$ and a nontrivial multiplicative character $\chi$ of $\F_{q^2}$ such that
\begin{equation}\label{eq:tech}
\frac{1}{k}
\bigg|\sum_{j=1}^{k} \chi(x_j)\bigg| \geq \varepsilon,  
\end{equation}
whenever $x_1,x_2, \ldots, x_k \in S$;
\item $q=p^n$, where $p>4.1n^2/\epsilon^2$.
\end{enumerate}
Assuming $d>1$ and $d\mid (q+1)$, by taking $\epsilon=1$ and $\chi$ a multiplicative character of $\F_{q^2}$ with order $d$, this recovers the result by Blokhuis and Sziklai on $GP(q^2,d)$ provided that $q=p^n$ with $p>4.1n^2$. While this approach successfully extended Blokhuis' theorem to a larger family of Cayley graphs, inequality~\eqref{eq:tech} in the first assumption is generally difficult to verify. Also, the second assumption is undesired (since there are infinitely many $q$ for which the assumption fails), although it is less important. Another limitation of this approach is that it relies on deep results from number theory. 

In this paper, inspired by several ideas from \cite{AY22, S20, Y25+}, we develop a new approach to prove Blokhuis' theorem and its extension. Our proof combines tools from finite geometry and additive combinatorics, and avoids the use of heavy machinery from number theory from \cite{AY22, Y22}. In particular, we give a short proof of Blokhuis' theorem, Sziklai's theorem, as well as a result of the author \cite{Y22} concerning a converse of Sziklai's theorem.

Our main result is the following.
\begin{thm}\label{thm:main}
Let $S \subseteq \F_{q^2}^*$ with $S=-S$ such that 
$$
\big|SSSS^{-1}S^{-1}S^{-1}\big|\leq \frac{q^2-3}{2} \quad \text{or} \quad \big|SS^{-1}S^{-1}\F_q^*\big|\leq \frac{q^2-1}{2}.
$$
If $A\subseteq \F_{q^2}$ with $|A|=q$ and $0,1\in A$ such that $A-A \subseteq S \cup \{0\}$, then $A$ is the subfield $\F_q$. Equivalently, if $A\subseteq \F_{q^2}$ is a clique in $X=\Cay(\F_{q^2}; S)$ with $|A|=q$ and $0,1\in A$, then $A$ is the subfield $\F_q$.   
\end{thm}

Inequality~\eqref{eq:tech} above from \cite{AY22} attempts to measure the multiplicative structure of the connection set $S$ analytically. Instead, in our approach, to measure the richness of the multiplicative structure of $S$, we follow the standard approach in additive combinatorics, namely, study the (multiplicative) doubling constant $C=|SS|/|S|$. If $C=1$, then $S$ is closed under multiplication and in particular $S$ is a subgroup of $\F_{q^2}^*$. In general, if $C$ is small, then $S$ is ``close to" a subgroup of $\F_{q^2}^*$ by the Freiman theorem for groups proved by Green and Ruzsa~\cite{GR07}. In our setting, the Pl\"unnecke–Ruzsa inequality (see for example \cite[Theorem 1.2]{P12}) implies that $|SSSS^{-1}S^{-1}S^{-1}|\leq C^6|S|$. Thus, we obtain the following corollary immediately, which shows that Theorem~\ref{thm:main} applies if $S$ has small doubling and $|S|$ is not too large.

\begin{cor}\label{cor1}
Let $S \subseteq \F_{q^2}^*$ with $S=-S$. Assume that $C^6|S|\leq \frac{q^2-3}{2}$, where $C=|SS|/|S|$. If $A\subseteq \F_{q^2}$ is a clique in $X=\Cay(\F_{q^2}; S)$ with $|A|=q$ and $0,1\in A$, then $A$ is the subfield $\F_q$. 
\end{cor}

Under the assumption in Theorem~\ref{thm:main}, the only possible clique in $X$ with size $q$ containing $0,1$ is the subfield $\F_q$. To check whether $\F_q$ indeed forms a clique, it suffices to verify if $\F_q=\F_q-\F_q \subseteq S \cup \{0\}$. Thus, under some mild assumption on $S$, Theorem~\ref{thm:main} implies the following dichotomy: if $\F_q \subseteq S \cup \{0\}$, then all maximum cliques arise from the subfield $\F_q$; otherwise, $\omega(X)\leq q-1$. Next, we consider these two cases separately in a more precise manner. 

\medskip

Let $X=\Cay(\F_{q^2}; S)$. Assume that $S$ is a proper subgroup of $\F_{q^2}^*$ such that $\F_q^* \subseteq S$ so that $\F_q$ is a clique in $X$. Since $\F_q^* \subseteq S$ and $S$ is a proper subgroup of $\F_{q^2}^*$, we have $|SS^{-1}S^{-1}\F_q^*|=|S|\leq \frac{q^2-1}{2}$ and thus Theorem~\ref{thm:main} applies. As a consequence, we recover Blokhuis' theorem and Sziklai's theorem immediately. 

\begin{cor}[Blokhuis, Sziklai]\label{cor:BS}
Let $d\geq 2$ be an integer and $q$ be a prime power with $d\mid (q+1)$. Then the only maximum clique in $GP(q^2,d)$ containing $0,1$ is the subfield $\F_q$.    
\end{cor}

Next we consider generalized Paley graphs $GP(q^2,d)$ with $d\nmid (q+1)$. In this case, the main result in \cite{AY22} does not apply: it is easy to verify that the subfield $\F_q$ does not form a clique \cite[Lemma 2.2]{Y22} and there is no obvious large clique. Thus, it is tempting to predict that $\omega(GP(q^2,d))\leq q-1$. This was confirmed by the author \cite{Y22} via a Fourier analytic argument, which complements Sziklai's theorem. The proof in \cite{Y22} crucially relied on several results from algebraic number theory, namely special properties of pure Gauss sums and semi-primitive Gauss sums, including Stickelberger's theorem \cite[Theorem 11.6.3]{BEW}, Hasse-Davenport lifting Theorem \cite[Theorem 11.5.2]{BEW}, and Evans' theorem~\cite{E81}. Given the nature of the proof, it only worked for generalized Paley graphs and cannot be extended to general Cayley graphs. 
Theorem~\ref{thm:main} quickly recovers this result; moreover, it can be used to strengthen this result. We illustrate one possible strengthening in the theorem below, where the connection set $S$ is the union of several cosets of a fixed subgroup of $\F_{q^2}^*$. 

\begin{thm}\label{thm:GP}
Let $H$ be a subgroup of $\F_{q^2}^*$ with index $d\geq 2$. Let $k$ be a nonnegative integer, $g$ be a generator of $\F_{q^2}^*$, $S=\bigcup_{j=0}^k g^j H$. Assume that $S=-S$. Let $X=\Cay(\F_{q^2};S)$.
\begin{enumerate}
    \item If $d \mid (q+1)$ and $d\geq 6k+2$, then $\omega(X)=q$ and all maximum cliques are of the form $a\F_q+b$, where $a\in S$ and $b\in \F_{q^2}$.
    \item If $d \nmid (q+1)$, $d\geq 12k+3$, and $q^2-1\geq 2d$, then $\omega(X)\leq q-1$.
\end{enumerate}
\end{thm}

In particular, when $k=0$, Theorem~\ref{thm:GP}(1) recovers Corollary~\ref{cor:BS} and Theorem~\ref{thm:GP}(2) recovers the main result in \cite{Y22}. We note that under the additional assumption that the characteristic $p$ is sufficiently large compared to $d$ and $n$, where $q=p^n$, Theorem~\ref{thm:GP}(1) has been further strengthened in \cite[Theorem 2.20]{AY22} by establishing inequality~\eqref{eq:tech}.

\medskip

\textbf{Notations.} We follow standard notations for arithmetic operations among sets.  Given two subsets $A$ and $B$ of a finite field $\F_r$, we write $A+B=\{a+b: a\in A, b\in B\}$, $-A=\{-a:a \in A\}$, $AB=\{ab: a \in A, b \in B\}$, and $A^{-1}=\{a^{-1}: a \in A\}$ if $0\notin A$. We say $A$ is a \emph{subspace} in $\F_r$ if $A$ is closed under addition \footnote{Since $\F_r$ is a finite field, if $A\subseteq \F_r$ is closed under addition, then $A$ is a vector space over $\F_p$, where $p$ is the characteristic of $\F_r$}.

\section{Directions determined by a point set in an affine plane}\label{prelim}
Let $AG(2,r)$ denote the {\em affine Galois plane} over the finite field $\F_r$. Let $U \subseteq AG(2,r)$. We use Cartesian coordinates in $AG(2,r)$ so that $U=\{(x_i,y_i):1 \leq i \leq |U|\}$.
The set of {\em directions determined by} $U \subseteq AG(2, r)$ is 
\[ \mathcal{D}_U=\left\{ \frac{y_j-y_i}{x_j-x_i} \colon 1\leq i <j \leq |U| \right \} \subseteq PG(1,r) \cong \F_r \cup \{\infty\},\]
 where $\infty$ is the vertical direction. The theory of directions is well-studied; we refer to R\'edei's book \cite{LR73} and Sz\H{o}nyi's survey \cite{S99}. The connection with the theory of directions and Cayley graphs has been explored in \cite{AY22, DSW, GY24, Y25+}.

R\'edei's seminal result \cite{LR73} states if $p$ is a prime and $U \subseteq AG(2,p)$ with $|U|=p$, then either all points in $U$ are collinear, or $|\mathcal{D}_U|\geq \frac{p+3}{2}$. For a general prime power $r$, one expects that a point set $U \subseteq AG(2,r)$ with $|U|=r$ determines many directions, unless $U$ admits some obvious algebraic structure. This can be made precise in the following celebrated result due to Blokhuis, Ball, Brouwer, Storme, and Sz{\H{o}}nyi \cite{Ball03, BBBSS}.
\begin{thm}[Blokhuis, Ball, Brouwer, Storme, and Sz{\H{o}}nyi]\label{linear}
 Let $U \subseteq AG(2,r)$ be a point set with $r$ points such that $(0,0) \in U$. Let $p$ be the characteristic of $\F_r$. 
 If $|\mathcal{D}_U|\leq \frac{r+1}{2}$, then $U$ is $\F_p$-linear, that is, for each $\beta \in \F_{r^2} \setminus \F_r$, $W=\{x+\beta y: (x,y) \in U\}$ forms a subspace in $\F_{r^2}$. 
\end{thm}
\begin{proof}
This is a special case of a result by Ball \cite{Ball03}. Note that the main theorem in \cite{Ball03} assumes that $U$ is the graph of the function $f\colon \F_r \to \F_r$. However, it is well-known that any set of $r$ points in $AG(2,r)$ that does not determine all the directions is affinely equivalent to the graph of a function (see for example \cite[page 342]{Ball03}).
\end{proof}

Next, we use Theorem~\ref{linear} to deduce some useful consequences.

\begin{prop}\label{prop:main}
Let $S \subseteq \F_{q^2}^*$ with $S=-S$ and $|SS^{-1}|\leq \frac{q^2-3}{2}$. If $A\subseteq \F_{q^2}$ is a clique in $X=\Cay(\F_{q^2}; S)$ with $|A|=q$ and $0\in A$, then $A$ is a subspace in $\F_{q^2}$.
\end{prop}
\begin{proof}
Since $A$ is a clique, we have $A-A \subseteq S\cup \{0\}$. Let $U= A \times A \subseteq AG(2,q^2)$; then 
$|U|=|A|^2=q^2$ and
$$
\mathcal{D}_U=\frac{A-A}{A-A}\subseteq \frac{S \cup \{0\}}{S \cup \{0\}} \subseteq SS^{-1} \cup \{0, \infty\}.
$$
Since $|SS^{-1}|\leq \frac{q^2-3}{2}$, it follows that $|\mathcal{D}_U|\leq \frac{q^2+1}{2}$. Fix $\beta\in \F_{q^4}\setminus \F_{q^2}$. Theorem~\ref{linear} then implies that $W=\{x+\beta y: (x,y) \in U\}$ forms a subspace in $\F_{q^4}$. 

Let $x,y\in A$. We have $(x,x),(y,y)\in U$ so that $(1+\beta)x, (1+\beta)y\in W$. Since $W$ is a subspace, $(1+\beta)(x+y)\in W$. Since  $1$ and $\beta$ are linearly independent over $\F_{q^2}$, it follows that $x+y\in A$. Thus, $A$ is a subspace, as required. 
\end{proof}

When $S$ is assumed to be the union of $\F_q^*$-cosets in $\F_{q^2}^*$, the assumption in the above proposition can be weakened.
\begin{prop}\label{prop:main2}
Let $S \subseteq \F_{q^2}^*$ with $S=-S$ and $|S\F_q^*|\leq \frac{q^2-1}{2}$. If $A\subseteq \F_{q^2}$ is a clique in $X=\Cay(\F_{q^2}; S)$ with $|A|=q$ and $0\in A$, then $A$ is a subspace in $\F_{q^2}$.
\end{prop}
\begin{proof}
Let $S'=S\F_q^*$. Since $|S'|\leq \frac{q^2-1}{2}$, $S'$ can be written as the union of at most $\frac{q+1}{2}$ many $\F_q^*$-cosets in $\F_{q^2}^*$. Since $S \subseteq S'$, $A$ remains to be a clique in the Cayley graph $X'=\Cay(\F_{q^2}; S')$. The proposition then follows from \cite[Theorem 1.2]{AY22}, which is a consequence of Theorem~\ref{linear}. 
\end{proof}

\section{Proof of main results}
For convenience, in this section, we use the following non-standard notation: if $A\subseteq \F_{q^2}$ and $0\in A$, we write $$A^{-1}=\{a^{-1}: a \in A\setminus \{0\}\} \cup \{0\}.$$

\subsection{A criterion for subfields}
The goal of this subsection is to prove the following proposition, which plays a crucial role in the proof of Theorem~\ref{thm:main}.

\begin{prop}\label{prop:subfield}
Assume that $A \subseteq \F_{q^2}$ such that $|A|=q$ and $0,1\in A$. If $A$ and $A^{-1}$ are both subspaces in $\F_{q^2}$, then $A$ is the subfield $\F_q$.
\end{prop}
\begin{proof}
Our key observation is the following claim. 
\begin{claim}\label{claimx^2}
If $x,y\in A$ with $y\neq 0$, then $x^2y^{-1}\in A$. In particular, if $x\in A$, then $x^2\in A$.        
\end{claim}
\begin{poc}
If $x=0$ or $x=y$, then clearly $x^2y^{-1}=x\in A$. Next assume that $x\neq 0$ and $x\neq y$. Then we have $x^2+xy=x(x+y)\neq 0$. Since $x+y\in A$ and $A^{-1}$ is a subspace, we have 
$$
\frac{1}{x}-\frac{1}{x+y}=\frac{y}{x^2+xy}=\frac{1}{x^2y^{-1}+x}\in A^{-1}.
$$
It follows that $x^2y^{-1}+x\in A$ and thus $x^2y^{-1}\in A$.
\end{poc}

It suffices to show that $A\setminus \{0\}$ is subgroup of $\F_{q^2}^*$. Indeed, it then follows that $A$ is a field, which implies that $A=\F_q$ since $A \subseteq \F_{q^2}$ and $|A|=q$. Thus, by the subgroup test, it suffices to show for each $x,y\in A \setminus \{0\}$, we have $xy^{-1} \in A$. 

First, assume that $q$ is even. In this case, note that the map $x \mapsto x^2$ is a field automorphism of $\F_{q^2}$. In particular, Claim~\ref{claimx^2} implies that $A=\{x^2:x\in A\}$. Thus, Claim~\ref{claimx^2} implies that for each $x,y\in A \setminus \{0\}$, we have $xy^{-1} \in A$, as required. 

Next, assume that $q$ is odd. Let $x,y\in A$. Then $x+y,x-y\in A$. By Claim~\ref{claimx^2}, we have $(x+y)^2, (x-y)^2\in A$. It follows that $4xy=(x+y)^2-(x-y)^2\in A$ and thus $xy\in A$. This shows that $A$ is closed under multiplication. Also, Claim~\ref{claimx^2} implies that $A=A^{-1}$. Thus, for each $x,y\in A \setminus \{0\}$, we have $xy^{-1} \in A$, as required. 
\end{proof}

\begin{rem}
After submitting the manuscript, we became aware that stronger versions of Proposition~\ref{prop:subfield} have appeared in \cite[Corollary 1.4]{C13} as well as \cite{M14, M15}. Here we included a short proof of Proposition~\ref{prop:subfield} for the sake of completeness. 
\end{rem}

\subsection{Proof of Theorem~\ref{thm:main}}

Assume that $A$ is a clique of size $q$ in $X$ with $0,1\in A$. For each $x\in A \setminus \{0\}$, we have $x\in S$ and thus
$$
\frac{1}{x}-0=\frac{1}{x}\in S^{-1}.
$$
For $x,y\in A\setminus \{0\}$ with $x\neq y$, we have $x,y,y-x\in S$ and thus
$$
\frac{1}{x}-\frac{1}{y}=\frac{y-x}{xy} \in SS^{-1}S^{-1}.
$$
Note that $S^{-1} \subseteq SS^{-1}S^{-1}$. Thus, $A^{-1}$ is a clique in the Cayley graph $X'=\Cay(\F_{q^2}; T)$ with $T=SS^{-1}S^{-1}$. Now $$TT^{-1}=SS^{-1}S^{-1} (SS^{-1}S^{-1})^{-1}=SSSS^{-1}S^{-1}S^{-1}.$$ Next we consider two cases:

(1) Assume that $|SSSS^{-1}S^{-1}S^{-1}|=|TT^{-1}|\leq \frac{q^2-3}{2}$. Note that $SS^{-1} \subseteq TT^{-1}$. By Proposition~\ref{prop:main} applied to the clique $A$ in the Cayley graph $X$ and to the clique $A^{-1}$ in the Cayley graph $X'$, we deduce that $A$ and $A^{-1}$ are both subspaces in $\F_{q^2}$. Proposition~\ref{prop:subfield} then implies that $A=\F_q$. 

(2) Assume that $|SS^{-1}S^{-1}\F_q^*|=|T\F_q^*|\leq \frac{q^2-1}{2}$. Note that $S\F_q^*\subseteq T\F_q^*$. Applying Proposition~\ref{prop:main2} instead, we deduce that $A=\F_q$ similarly. 

\subsection{Proof of Theorem~\ref{thm:GP}}
Since $g$ is a generator of $\F_{q^2}^*$, $H$ is the subgroup of $\F_{q^2}^*$ of index $d$, and $\F_{q}^*$ is the subgroup of $\F_{q^2}^*$ of index $(q+1)$, it follows that $g^d$ is a generator of $H$ and $g^{q+1}$ is a generator of $\F_{q}^*$. 

(1) Since $d\mid (q+1)$, $\F_q^*$ is a subgroup of $H$ and thus $a\F_q+b$ forms a clique in $X$ for each $a\in S$ and $b\in \F_{q^2}$. Since $S=\bigcup_{j=0}^k g^j H$, we have
$$
SS^{-1}S^{-1}\F_q^*=SS^{-1}S^{-1}=\bigcup_{j=-2k}^{k} g^j H.
$$
Since $d\geq 6k+2$, it follows that
\begin{align*}
\big|SS^{-1}S^{-1}\F_q^*\big|
&=(3k+1)|H|=\frac{(3k+1)(q^2-1)}{d}\leq \frac{q^2-1}{2}.  
\end{align*}

Let $A$ be a clique in $X$ with size $q$. Since $X$ is a Cayley graph, without loss of generality, we can assume that $0\in A$. Since $q\geq 2$, we can take $\alpha \in A\setminus \{0\} \subseteq S$. Then $B=\alpha^{-1}A$ is a clique in $X'=\Cay(\F_{q^2};T)$, where $T=\alpha^{-1}S$. Note that $$|TT^{-1}T^{-1}\F_q^*|=\big|SS^{-1}S^{-1}\F_q^*\big|\leq \frac{q^2-1}{2}$$ and $0,1\in B$. Thus, by applying Theorem~\ref{thm:main} to the clique $B$ in $X'$, we deduce that $B=\F_q$ and thus $A=\alpha \F_q$ for some $\alpha \in S$. We conclude that $\omega(X)=q$ and all maximum cliques are of the form $a\F_q+b$ with $a\in S$ and $b\in \F_{q^2}$. 

(2) Let $d'=\gcd(d,q+1)$. Since $d\nmid (q+1)$, we have $d'\leq d/2$. It follows that $\F_q^* \cap H$ is the subgroup generated by $g^{d(q+1)/d'}$. Thus, for each $0\leq i\leq d/d'-1$, we have $$|\F_q^* \cap g^{id'} H|=\frac{(q^2-1)}{d(q+1)/d'}=\frac{(q-1)d'}{d}.$$ In particular, $\F_q^*$ can be written as the following disjoint union of sets with equal size:
\begin{equation}\label{eq2}
\F_q^*=\bigcup_{i=0}^{d/d'-1} (\F_q^* \cap g^{id'} H).    
\end{equation}

Since $S=\bigcup_{j=0}^k g^j H$, we have
$$
SSSS^{-1}S^{-1}S^{-1}=\bigcup_{j=-3k}^{3k} g^j H.
$$
Since $d\geq 12k+3$, it follows that
\begin{align*}
\big|SSSS^{-1}S^{-1}S^{-1}\big|
&=(6k+1)|H|=\frac{(6k+1)(q^2-1)}{d}\leq \frac{(d-1)(q^2-1)}{2d}\\
&=\frac{q^2-1}{2}-\frac{q^2-1}{2d}\leq \frac{q^2-1}{2}-1=\frac{q^2-3}{2}.   
\end{align*}

Assume otherwise that $\omega(X)\geq q$, then there is a clique $A$ in $X$ with size $q$. Without loss of generality, assume that $0\in A$. Since $q\geq 2$, we can take $\alpha \in A\setminus \{0\}$; say $\alpha \in g^{\ell}H$ for some integer $\ell$. Then $B=\alpha^{-1}A$ is a clique in $X'=\Cay(\F_{q^2}; T)$, where \begin{equation}\label{eq:T}
T=\alpha^{-1}S=g^{-\ell}S=\bigcup_{j=-\ell}^{k-\ell} g^j H.    
\end{equation}
Note that $$TTTT^{-1}T^{-1}T^{-1}=SSSS^{-1}S^{-1}S^{-1}$$
and $0,1\in B$. Thus, by applying Theorem~\ref{thm:main} to the clique $B$ in $X'$, we deduce that $B=\F_q$ and thus $\F_q^*\subseteq T$. By equation~\eqref{eq:T}, $T$ intersects with at most $\lfloor k/d'\rfloor+1$ cosets of $H$ of the form $g^{id'}H$ with $0\leq i\leq d/d'-1$. On the other hand, equation~\eqref{eq2} implies that $\F_q^*$ intersects with exactly $d/d'$ cosets of $H$ of the form $g^{id'}H$ with $0\leq i\leq d/d'-1$. Since $\F_q^* \subseteq T$, it follows that $d/d'\leq \lfloor k/d'\rfloor+1$. Since $d'\leq d/2$, we have $k\geq d'$. However, since $d\geq 12k+3$, we have $$d/d'\geq (12k+3)/d'\geq 12\lfloor k/d'\rfloor+1>\lfloor k/d'\rfloor+1,$$ a contradiction.  We conclude that $\omega(X)\leq q-1$.

\section*{Acknowledgments}
The author thanks Esen Aksoy, Shamil Asgarli, Bence Csajb\'ok, Raghu Tej Pantangi, Ilya Shkredov, J\'ozsef Solymosi, and Ethan White for helpful discussions. The author also thanks the anonymous referees for their valuable comments and suggestions. The research of the author was supported in part by an NSERC fellowship.

\bibliographystyle{abbrv}
\bibliography{main}

\end{document}